\documentclass{article}

\usepackage{amsmath,amssymb,amsthm}
\usepackage[english]{babel}
\usepackage{mathrsfs}
\usepackage[T1]{fontenc}

\newtheorem{theorem}{Theorem}
\newtheorem{corollary}[theorem]{Corollary}
\newtheorem{proposition}[theorem]{Proposition}
\newtheorem{lemma}[theorem]{Lemma}

\newtheorem*{theorem*}{Theorem}
\newtheorem{definition}{Definition}

\theoremstyle{remark}
\newtheorem{remark}{Remark}

\newtheorem{example}{Example}

\newcommand{\HH}{\mathbb{H}}
\newcommand{\OO}{\mathbb{O}}
\newcommand{\C}{\mathbb{C}}
\newcommand{\N}{\mathbb{N}}
\newcommand{\re}{\mr{Re}}
\newcommand{\im}{\mr{Im}}
\newcommand{\R}{\mathbb{R}}

\newcommand{\tr}{\mr{tr}}
\newcommand{\is}{\mathbb{S}}
\newcommand{\isa}{\is_{\alpha}}
\newcommand{\Deltaa}{\Delta_{\alpha}}
\newcommand{\ra}{r_{\alpha}}

\newcommand{\oc}{\mathbb{O}}

\newcommand{\mr}{\mathrm}

\newcommand{\lra}{\longrightarrow}

\newcommand{\pr}{\prime}

\begin{document}

\title{Zeros of regular functions of quaternionic and octonionic variable: a division lemma and the cam\-shaft effect}
%\runningtitle{Zeros of regular functions}

\author{R.\ Ghiloni, A.\ Perotti
\thanks{Work partially supported by MIUR (PRIN Project ``Propriet\`a geometriche delle variet\`a reali e com\-ples\-se") and GNSAGA of INdAM}
\\Department of Mathematics\\
  University of Trento\\ Via Sommarive, 14\\ I--38123 Povo Trento ITALY\\
  perotti@science.unitn.it}

\maketitle

\begin{abstract} 
We study in detail the zero set of a slice regular function of a quaternionic or octo\-nionic variable. 
By means of a division lemma for convergent power series, we find the exact relation existing between the zeros of two octonionic regular functions and those of their product.
In the case of octonionic polynomials, we get a strong form of the fundamental theorem of algebra. We prove that the sum of the multiplicities of zeros equals the degree of the polynomial and obtain a factorization in linear polynomials.
\end{abstract}


Math.\ Subj.\ Class: {30C15, 30G35, 32A30}

\section{Introduction}
\label{sec:Introduction}
Let $\HH$ denote the skew field of quaternions over $\R$. Recently, Gentili and Struppa \cite{GeSt1,GeSt2} introduced the notion of \emph{slice regularity} (or \emph{Cullen regu\-larity}) for functions of a quaternionic variable. This notion do not coincide with the classical one of Fueter regularity (we refer the reader to \cite{Su} and \cite{GHS} for the theory of Fueter regular functions). In fact, the set of slice regular functions on a ball $B_R$ centered in the origin of $\HH$ coincides with that of all power series $\sum_iw^ia_i$ that converges in $B_R$. In particular, all the standard polynomials $\sum_{i=0}^n w^ia_i$ with right quaternionic coefficients, which fail to be Fueter regular, define slice regular functions on the whole space $\HH$.
The class of slice regular functions contains also the radially holomorphic functions introduced by Fueter 
%by means of holomorphic functions of one complex variable 
%\cite{F1935} 
(cf.\ \cite[\S11]{GHS}). 
%Fueter [?] used radially holomorphic functions obtained from holomorphic functions of one complex variable to get, via the application of the laplacian~$\Delta$, functions regular in his sense (see [?]). This technique was generalized to higher dimensions by Sce [?] and Qian [?].
Such functions fix every complex plane generated by the reals and by a unit imaginary quaternion, while the slice regular functions do not need to fulfil this property. Actually, radially holomorphic functions correspond to the very special class of those slice regular functions, which can be expanded into power series with real coefficients. The notion of slice regularity was extended to functions of an octonionic variable in \cite{GeStPreprint,GeStoStVl}
and to functions 
%defined on domains of the Euclidean space $\R^{n+1}$ 
with values in a Clifford algebra in \cite{GeSt4,CoSaSt1}.

In this paper, we study in detail the zero set of a slice regular function. For simplicity, in the sequel, we use the terminology ``regular function'' in place of ``slice regular function''.
 We obtain by different techniques some properties  already known (cf.\
\cite{PoSh,GeStPreprint,GeSt2,GeSto1,GeSt3,GeStVl}) and we extend others to the octonionic case. We find the exact relation existing between the zeros of two octonionic regular functions and those of their product.
In the case of octonionic polynomials, we obtain a strong form of the fundamental theorem of algebra.
The main tool we use is a division lemma, which generalizes to octonionic power series a result  proved by Beck \cite{Be} for quaternionic polynomials and by Ser\^odio \cite{Se1} for octonionic polynomials.
In the simpler case of regular polynomials and convergent power series with real (i.e.\ central) coefficients, the properties of the zero set have been studied also by other authors~(cf.\ \cite{DaN}).

We describe in more details the structure of the paper.
In Section \ref{sec:Division}, we recall some basic definitions and prove the division lemma (Lemma \ref{division-lemma}). Given a regular function $f$ and an octonion $\alpha$, 
%upon division w.r.t.\ the minimal polynomial of $\alpha$ 
we can associate to $f$ a remainder, which is a linear or constant regular polynomial. This remainder describes completely the intersection of the zero set $V(f)$ of $f$ with the conjugacy class of $\alpha$.
We can then obtain easily  a structure theorem for $V(f)$ (Theorem \ref{thm4}). This result was proved for quaternionic polynomials by Pogorui and Shapiro \cite{PoSh}, for quaternionic regular functions by Gentili and Stoppato in \cite{GeSto1} and it was extended to octonionic regular functions in \cite{GeStVl}.
The division lemma and the concept of \emph{normal series} associated to a regular function (see Section \ref{sec:Division} for precise definitions) suggest the definition of the multiplicity of a zero of a regular function.  Our definition is equivalent on quaternionic polynomials with the one given in \cite{BrW} and in \cite{GeSt3}.

In Section \ref{sec:Zerosofproducts}, we describe the zeros of the product $f\ast g$ of two regular functions in terms of the zeros of $f$ and $g$. 
We prove (Lemma \ref{Nfg}) that the normal series of a regular function is multiplicative, a result that is non--trivial in the octonionic case. An immediate consequence is that the conjugacy classes of the zeros of $f\ast g$ are exactly those containing zeros of $f$ or $g$. The precise relation between such zeros is stated in Theorem \ref{thm:cam-shaft}. We call this relation the ``camshaft effect''. In the associative case, it reduces to known results (cf.\ \cite{GeSto1} for quaternionic regular functions and \cite{Lam,SeSiu} for polynomials).

Section \ref{sec:Applications} contains some applications of the preceding results to regular polynomials. The fundamental theorem of algebra was proved by Niven \cite{Ni} for standard quaternionic polynomials. It was extended to other polynomials over $\HH$ by Eilenberg and Niven \cite{EiNi} and to octonionic polynomials by Jou \cite{Jou}.
In the book \cite[pp. 308ff]{EiSt}, Eilenberg and Steenrod gave a topological proof of the theorem valid for a class of real algebras including the complex numbers, the quaternions and the octonions. See also \cite{To}, \cite{RoO} and \cite{GeStVl} for other proofs.

In this context, our aim is a strong form of the fundamental theorem, in which a formula for the sum of the multiplicities of zeros is achieved.
Gordon and Motzkin \cite{GoM} proved, for polynomials on a (associative) division ring, that the number of conjugacy classes containing zeros of $f$ cannot be greater than the degree $n$ of $f$. This estimate was improved on the quaternions by Pogorui and Shapiro \cite{PoSh}: if $f$ has $m$ spherical zeros and $l$ non--spherical zeros, then $2m+l\le n$. 
Gentili and Struppa \cite{GeSt3} showed that, using the right definition of multiplicity, the number of zeros of $f$ equals the degree of the polynomial.
We generalize this strong form to the octonions (Theorem \ref{TFA}), giving a proof inspired by the simple argument used in \cite{PoSh}. From this result and the division lemma, we get a factorization lemma for regular polynomials (Lemma \ref{factlemma}) and some sufficient conditions for the finiteness of $V(f)$.
%The factorization lemma, together with the formulas of the ``camshaft effect'', gives an iterative method to construct an octonionic polynomial with a prescribed set of zeros (cf.\ \cite{Se2}).

The definitions and results are stated over octonions only, but they remain valid over quaternions as well. Since $\HH$ can be identified with a real subalgebra of the octonions, the corresponding statements for quaternions can be obtained directly from their octonionic version by specializing coefficients and variables.

%%%%%%%%%%%%%%%%%%%%%

\section{Division of regular series and multiplicity of zeros} \label{sec:Division}
Let $\oc$ be the non--associative division algebra of octonions (also called Cayley numbers). We refer to \cite{Sch,Lam,numbers} for the main properties of the algebras $\HH$ and~$\OO$.
 We recall that $\OO$ can be obtained from $\HH$ by the Cayley--Dickson process. Any $x\in\OO$ can be written as $x=x_1+x_2k$, where $x_1,x_2\in\HH$ and $k$ is a fixed imaginary unit of $\OO$. The addition on $\OO$ is defined componentwise and the product is defined by
\[xy=(x_1+x_2k)(y_1+y_2k)=x_1y_1-\bar y_2x_2+(x_2\bar y_1+y_2x_1)k.\]
Let $\{1,i,j,ij\}$ denote a real basis of $\HH$. Then a basis of the real algebra $\OO$ is formed by $\{1,i,j,ij,k,ik,jk,(ij)k\}$.

Let $B_R$ be the open ball of $\oc$ centered in the origin with (possibly infinite) radius $R$. Let $f:B_R \lra \oc$ and $g:B_R \lra \oc$ be regular functions with series expansions $f(w)=\sum_iw^ia_i$ and $g(w)=\sum_jw^jb_j$. 
The usual product of polynomials, where $w$ is considered to be a commuting variable (cf.\ for example \cite{Lam} and \cite{GRW,GGRWL}), can be extended to power series, and hence to slice regular functions (cf.\ \cite{GeSto1,GeSt3}).
Recall that $f \ast g$ is the regular function defined on $B_R$ by setting
$$\textstyle
(f \ast g)(w):=\sum_kw^k\big(\sum_{i+j=k}a_ib_j\big).
$$
We denote by $\overline{f}$ the  regular function defined on $B_R$ by 
$\overline{f}(w):=\sum_iw^i\overline{a}_i$.
If $f$ has real coefficients, i.e.\ $f=\bar f$, we will say that $f$ is \emph{real}. In this case, $(f\ast g)(\alpha)=f(\alpha)g(\alpha)$ for every $g$ and every $\alpha\in\oc$ and we will write $f g$ in place of $f\ast g$.
Moreover, if $f$ is real then $(f\ast g)\ast h=(fg)\ast h=f(g\ast h)$ for every $g,h$, and $\overline{f\ast g}$ is equal to $f \overline g$.

We denote by $N(f)$ the \emph{normal series} (or \emph{symmetrized series}) \emph{of $f$} defined by
\[N(f):=f \ast \overline{f}=\overline{f}\ast f.\]
We remark that $N(f)$ is regular and real on $B_R$.
 
Let $\alpha \in \oc$. 
%We say that $\alpha$ is \textit{non--real} if $\alpha \not\in \R$. 
We denote the trace $\tr(\alpha)=\alpha+\overline{\alpha}$ of $\alpha$ also by $t_{\alpha}$, the squared norm of $\alpha$ by $n_{\alpha}$ and define the real polynomial $\Delta_{\alpha}$, called \emph{characteristic polynomial of $\alpha$}, by
$$
\Delta_{\alpha}(w):=N(w-\alpha)=(w-\alpha) \ast (w-\overline{\alpha})=w^2-w \cdot t_{\alpha}+n_{\alpha}.
$$
It is well--known (cf.\ \cite{Se1}) that $\alpha$ is conjugate to an octonion $\beta$ if and only if $\Delta_{\alpha}=\Delta_{\beta}$.
% or, equivalently, if and only if $\is_{\alpha}=\is_{\beta}$. 
Indicate by 
%$\mk{a}$ the conjugacy  class of $\alpha$: $\mk{a}=\is_{\alpha}
$\isa$ the conjugacy  class of $\alpha$: $\isa:=\{\beta\in\oc\ |\ \beta=\re(\alpha)+|\im(\alpha)|\,I, I\in\is \}$, where $\is=\{I\in\oc\ |\ I^2=-1\}$ is the sphere of imaginary units. 
Note that $\isa$ reduces to the point $\{\alpha\}$ when $\alpha$ is real and it is a six--dimensional sphere when $\alpha$ is non--real.

We denote by $V(f)$ the zero set of $f$.
If $\alpha\in V(f)$ is real, we call $\alpha$ a \emph{real zero of $f$}.
If $\alpha$ is non--real and $\isa \subset V(f)$, we call $\alpha$ a \emph{spherical zero of $f$}.  Otherwise, we call $\alpha\in V(f)$ an \emph{isolated zero of $f$}. This terminology is justified by  Theorem \ref{thm4} below.

%For convenience, define $\is_{\mk{a}}:=\is_{\alpha}$, $\mr{tr}(\mk{a}):=\mr{tr}(\alpha)$, $|\mk{a}|:=|\alpha|$ and $\Delta_{\mk{a}}:=\Delta_{\alpha}$.

Let us present our division lemma.

\begin{lemma}\label{division-lemma}
Let  $f:B_R \lra \oc$ regular and let $\alpha\in B_R$. The following statements hold.
\begin{itemize}
\item[$(1)$]
There exist, and are unique, a regular function $g:B_R \lra \oc$ and an octonion $r\in \oc$ such that $f(w)=(w-\alpha)*g(w) +r$. 
\item[$(2)$]
Let $\alpha\in\oc$ be non--real. There exist, and are unique, a regular function $h:B_R \lra \oc$ and octonions $a,b \in \oc$ such that $f(w)=(\Delta_{\alpha} \, h)(w)+w \, a+b$. 
\end{itemize}
\end{lemma}
\begin{proof}
(1)\ Let us prove the existence of $g$ and of $r$. First, observe that, if such $g$ and $r$ exist, then $r$ must be equal to $f(\alpha)$. In fact, it is easy to see that the value of the function $(w-\alpha) \ast g(w)$ in $\alpha$ is zero. Assume that $f=\sum_iw^ia_i$ has positive convergence radius $R$. Recall that
\[R=\left(\limsup_{n\rightarrow +\infty}\sqrt[n]{|a_n|}\right)^{-1}.\]
Let $\alpha\in B_R$. If $\alpha=0$, then we can write
$f=w\left(\sum_{i=0}^{+\infty}w^ia_{i+1}\right)+a_0.$
Suppose $\alpha \neq 0$. A formal series computation imposes to define the coefficients of the power series $g=\sum_n w^n b_n$ as follows:
\[b_n:=\alpha^{-1-n}\left(f(\alpha)-\sum_{j=0}^{n}\alpha^j a_j\right)=\alpha^{-1-n}\sum_{j=n+1}^{+\infty}\alpha^j a_j\]
for each $n \in \N$.
Let us show that the series $g$ has radius of convergence at least~$R$. Since $|\alpha|<R$, $\limsup_{n\rightarrow +\infty}\sqrt[n]{|a_n|}<|\alpha|^{-1}$. This implies that, for every fixed $\rho$ with $|\alpha|<\rho<R$, there exists an integer $n_{\rho}$ such that the inequality $|a_j|\le \rho^{-j}$ is satisfied for every $j> n_{\rho}$. Let $x\in B_R$ and let $\rho$ be chosen in such a way that $\max\{|\alpha|,|x|\}<\rho<R$. 
 Then, for every $n\ge n_{\rho}$, it holds:
\begin{align*}
|x^n b_n|&\le|x|^n|b_n|\le|x|^n|\alpha|^{-1-n}\sum_{j=n+1}^{+\infty}|\alpha|^j |a_j|\le
|x|^n|\alpha|^{-1-n}\sum_{j=n+1}^{+\infty}|\alpha|^j \rho^{-j}\\
&=|x|^n|\alpha|^{-1-n}\left(\frac{|\alpha|}{\rho}\right)^{n+1}\frac1{1-|\alpha|/\rho}
=\left(\frac{|x|}{\rho}\right)^n \frac{1}{\rho-|\alpha|}.
\end{align*}
Therefore, the power series $\sum_n w^n b_n$ converges at every $x\in B_R$.
It remains to prove the uniqueness of $g$ and of $r$. We have just seen that $r$ must be equal to $f(\alpha)$ so the octonion $r$ is uniquely defined by $f$ and $\alpha$. If $f(w)=(w-\alpha)*g(w) +r=(w-\alpha)*g'(w) +r$, then $(w-\alpha)\ast (g(w)-g'(w))=0$ and this easily implies that $g=g'$.

%\noindent\smallskip
(2)\ Let $\alpha$ be non--real. Using the first part twice, we get $s\in\oc$ and a regular function $h$ on $B_R$ such that
\begin{align*}
f(w)&=(w-\alpha)\ast g(w)+r=(w-\alpha)\ast ((w-\bar\alpha)\ast h(w)+s)+r\\
&=(\Deltaa h)(w)+w s-\alpha s+r.
\end{align*}
To prove uniqueness, assume that  $f(w)=\Delta_{\alpha}  h+w  a+b=\Delta_{\alpha}  h'+w  a'+b'$. Then $f(\alpha)=\alpha a+b=\alpha a'+b'$, $f(\bar\alpha)=\bar\alpha a+b=\bar\alpha a'+b'$, from which we get $(\alpha-\bar\alpha)(a-a')=0$ and then $a=a'$, $b=b'$. Finally, $\Deltaa h=\Deltaa h'$ gives $h=h'$ on $B_R\setminus\isa$. By a density argument, $h$ and $h'$ must coincide everywhere.
\end{proof}

\begin{remark} \label{rem:pol}
As a by--product of the proof of the preceding theorem, we obtain that, if $f$ is a regular octonionic polynomial of degree $n>0$, then $g$ is a regular octonionic polynomial of degree $n-1$.
\end{remark}

\begin{definition}
Let the regular functions $g,h$ and the octonions $r,a,b$ be as above. If $\alpha$ is real, we call $r$ the \emph{remainder of $f$ with respect to $\alpha$} and the function $g$ the \emph{quotient of $f$ w.r.t.\ $\alpha$}.
If $\alpha$ is non--real, we call \emph{remainder of $f$ with respect to $\alpha$} the octonionic polynomial  defined by\, $w \, a+b$ and  \emph{quotient of $f$ w.r.t.\ $\alpha$} the function $h$. Note that they depend only on the conjugacy class $\isa$ of $\alpha$.
In any case, we will denote the remainder by $\ra(f)$. 
\end{definition}

If $f$ is real, then the coefficients of the quotient function and of the remainder are real. This fact is a consequence of the uniqueness of quotient and remainder:
\[(w-\alpha)\, g +r=f=\bar f=(w-\alpha)\, \bar g +\bar r\text{\quad($\alpha$ real)}\]
or
\[\Delta_{\alpha} \, h+w \, a+b=f=\bar f=\Delta_{\alpha} \, \bar h+w \, \bar a+\bar b\text{\quad($\alpha$ non--real)}
\]
imply $g=\bar g$, $r=\bar r$, $h=\bar h$, $a=\bar a$ and $b=\bar b$.
%belong to the subalgebra $\langle 1,\alpha\rangle$ of $\oc$ generated by $1$ and $\alpha$, which is isomorphic to the complex field $\C$ if $\alpha$ is non--real. The same conclusion follows if the coefficients of $f$ lie in $\langle 1,\alpha\rangle$.

As a first application of Lemma \ref{division-lemma}, we obtain:

\begin{corollary}\label{cor2}
Let $\ra(f)$ be the remainder of $f$ with respect to $\alpha$. The following statements hold.
\begin{itemize}
 \item[$(1)$] Let $\alpha$ be real. Then $\alpha$ is a zero of $f$ if and only if\, $\ra(f)=0$; that is, $w-\alpha \mid f$.
 \item[$(2)$] Let $\alpha$ be non--real and let $\ra(f)(w)=w \, a+b$. Then we have:
  \begin{itemize}
    \item[$\mr{(i)}$] $\alpha$ is a spherical zero of $f$ if and only if\, $a=0=b$; that is, $\Deltaa \mid f$.
    \item[$\mr{(ii)}$] $\alpha$ is an isolated zero of $f$ if and only if\, $a \neq 0$ and $\alpha=-ba^{-1}$.
  \end{itemize}
In particular, either $\isa \subset V(f)$ or $\isa \cap V(f)$ is empty or $\isa \cap V(f)$ consists of a single point. The latter situation occurs if and only if $a\neq 0$ and $-ba^{-1} \in \is_{\alpha}$.
\end{itemize}
\end{corollary}
\begin{proof}
 By using Lemma \ref{division-lemma}, we can follow the same lines of proof as in the polynomial case (cf.\ \cite[\S3]{Se1}). 

If $\alpha\in\R$, the statement is trivial, since $f(\alpha)=\ra(f)$. 
 
Now assume that $\alpha$ is non--real. If $\ra(f)=0$, i.e. $a=b=0$, then $f(\beta)=\Deltaa(\beta)\cdot h(\beta)=0$ for every $\beta\in\is_{\alpha}$. If $a\ne0$ and $\alpha=-ba^{-1}$, then $f(\alpha)=\Deltaa(\alpha)\cdot h(\alpha)=0$. If $\is_{\alpha}$ contains two zeros $\beta$ and $\gamma$ of $f$, then the equality 
 \[\Delta_{\alpha}(\beta)\cdot h(\beta)+\beta \, a+b=f(\beta)=0=f(\gamma)=\Delta_{\alpha}(\gamma)\cdot h(\gamma)+\gamma\, a+b\]
 gives $\beta \, a=\gamma\, a$, since $\Delta_{\alpha}$ vanishes on $\beta,\gamma\in\is_{\alpha}$. If $\beta\ne\gamma$, then it must be $a=b=0$.
\end{proof}

If $f$ is real, then $\ra(f)=w \, a+b$ with $a,b$ real. Therefore $-ba^{-1} \in \R$ and $f$ has no isolated (non--real) zeros.
%$N(f)=f^2$. 
This means that the zeros of $f$ are all real or spherical and $V(f)$ is the union of the spheres $\isa$ with $\alpha\in V(f)$ (cf.\ \cite{DaN}).

If $f$ differs from a real $g$ by a constant non--real octonion, then we have the opposite situation: $f$ can have only isolated zeros. This property was proved for regular polynomials over $\HH$ in \cite[\S16.19]{Lam}.

\begin{proposition}\label{pro3}
Let $f(w)=g(w)+a_0$ be a regular function on $B_R$. Assume that $g$ is real and $a_0\in\oc$ is non--real. Then, if $V(f)$ is not empty, its elements are isolated zeros of $f$ contained in the complex plane generated by $1$ and $a_0$. 
\end{proposition}
\begin{proof} 
Let $\alpha\in V(f)$. Since $a_0=-g(\alpha)$ is non--real, also $\alpha$ must be non--real. 
From the division lemma (Lemma \ref{division-lemma}) applied to $g$, we get: $g(w)=\Deltaa h(w)+wa+b$, with real $a,b$. Then $\ra(f)=wa+b+a_0$ and $\alpha$ cannot be spherical. Otherwise, it would be $a_0=-b$. Therefore, $\alpha=-(b+a_0)a^{-1}$ is an isolated zero of $f$ belonging to the complex plane generated by $1$ and $a_0$.
\end{proof}

We still get the nonexistence of spherical zeros under a weaker condition on~$f$.

\begin{proposition}\label{pro4}
Let $f(w)=g(w)+wa_1+a_0$ be a regular function on $B_R$. Assume that $g$ is real and either $a_0$ or $a_1$ is non--real. Then $f$ has no spherical zeros.  
\end{proposition}
\begin{proof} 
We proceed as before. Assume that $V(f)\ne\emptyset$. Let $\alpha\in V(f)$. Then $\ra(f)=w(a+a_1)+b+a_0$, with $a,b$ real.
Then $\alpha$ cannot be spherical. Otherwise, it would be $a_0=-b$ and $a_1=-a$. 
\end{proof}

If $\alpha$ is real and $f(w)=(w-\alpha)\,g(w) +r$ for some regular function $g$ and $r \in \OO$, then $\ra(f)=r$,  $\ra(\bar f)=\bar r$ and  
\begin{equation}\label{fs-real}
N(f)=(w-\alpha)^2\, N(g)+(w-\alpha)(g \ast \bar r + r \ast \bar g)+n_r.
\end{equation}
It follows that $\ra(N(f))=n_r$. If $\alpha$ is non--real and $f$ has remainder\, $\ra(f)(w)=w \, a+b$, then we have that $\ra(\bar f)(w)=w \, \bar a+\bar b$ and hence
\[
\ra(N(f))=\ra(f \ast \bar f)(w)=w (a\bar b+b\bar a+t_{\alpha}n_a)+n_b-n_{\alpha}n_a.
\]
The latter equality follows immediately from the following fact
\begin{eqnarray*}
(wa+b) \ast (w\bar a+\bar b) &=& w^2 n_a+w(a\bar b+b\bar a)+n_b=\\
&=& (\Deltaa+w t_{\alpha}-n_{\alpha})n_a+w(a\bar b+b\bar a)+n_b.
\end{eqnarray*}

\begin{corollary}\label{cor3}
Given $\alpha \in \OO$, the following statements are equivalent: 
\begin{itemize}
 \item[$(1)$] $\is_{\alpha} \cap V(f)$ is non--empty.
 \item[$(2)$] $\is_{\alpha} \cap V(N(f))$ is non--empty.
 \item[$(3)$] $\Deltaa \mid N(f)$.
\end{itemize}
In particular, we have that \[V(N(f))=\bigcup_{\alpha \in V(f)}\is_{\alpha}.\]
\end{corollary}
\begin{proof}
If $\alpha$ is real and $\ra(f)=r$, then $\ra(N(f))=n_r=0$ if and only if $r=0$. If this is the case, Eq.\ \eqref{fs-real} implies that $\Deltaa=(w-\alpha)^2$ divides $N(f)$.
If $\alpha$ is non--real, then $\Deltaa \mid N(f)$ if and only if $\ra(N(f))=0$; that is,
\begin{equation}\label{rfs}
a\bar b+b\bar a+t_{\alpha}n_a=0,\quad n_b-n_{\alpha}n_a=0.
\end{equation}
If $a=b=0$, then $\is_{\alpha} \subseteq V(f)$ and equations \eqref{rfs} are trivial. If $a\ne0$ and $\beta=-ba^{-1}\in \is_{\alpha} \cap V(f)$, then $t_{\alpha}n_a=t_{\beta}n_a=\beta a\bar a+a\bar a\bar\beta=-b\bar a-a\bar b$. Moreover, $n_{\alpha}n_a=n_{\beta}n_a=n_b$ and equations \eqref{rfs} are satisfied.

Conversely, suppose that equations \eqref{rfs} are satisfied. If $a=0$, then $b=0$ and hence $\isa \subset V(f)$. Let $a\ne0$. Define the octonion $\beta:=-ba^{-1}$. Proceeding as above, we obtain that $t_{\alpha}n_a=t_{\beta}n_a$ and $n_{\alpha}n_a=n_{\beta}n_a$ or, equivalently, $t_{\alpha}=t_{\beta}$ and $n_{\alpha}=n_{\beta}$. It follows that $\beta \in \isa \cap V(f)$. Since $N(f)$ is real, its zeros are all real or spherical, so $V(N(f))$ is equal to $\bigcup_{\alpha \in V(f)}\is_{\alpha}$.
\end{proof}

From the preceding corollaries, we immediately get a new proof of the following result. It was proved for quaternionic polynomials by Pogorui and Shapiro \cite{PoSh}, for quaternionic regular functions by Gentili and Stoppato in \cite{GeSto1} and it was extended to octonionic regular functions in \cite{GeStVl}.

\begin{theorem}\label{thm4}
Let $f:B_R \longrightarrow \mathbb{O}$ be a regular function, which does not vanish on the whole $B_R$. For each $I \in \mathbb{S}$, denote by $\C_I$ the complex plane of $\OO$ generated by $1$ and $I$; that is, $\C_I:=\{x+yI \in \mathbb{O} \, | \, x,y \in \mathbb{R}\}$. The following statements hold.
\begin{itemize}
 \item[$(1)$] For all $I \in \mathbb{S}$, $\C_I \cap \bigcup_{\alpha \in V(f)}\mathbb{S}_{\alpha}$ is closed and discrete in $\C_I \cap B_R$.
 \item[$(2)$] For each $\beta \in \bigcup_{\alpha \in V(f)}\mathbb{S}_{\alpha}$, either\, $\mathbb{S}_{\beta} \subset V(f)$ or\, $\mathbb{S}_{\beta} \cap V(f)$ consists of a single point.
\end{itemize}
\end{theorem}

\begin{proof}
Since $N(f)$ is regular on $B_R$, and then holomorphic on $\C_I\cap B_R$, its zero--set $\C_I \cap \bigcup_{\alpha \in V(f)}\mathbb{S}_{\alpha}=\C_I\cap V(N(f))$ is closed and discrete in $\C_I\cap B_R$. The second statement follows from Corollary \ref{cor2}.
\end{proof}

\begin{definition}
Given a non--negative integer $s$ and an octonion $\alpha$ in $V(f)$, we say that $\alpha$ is a \emph{zero of $f$ of multiplicity $s$} if $\Deltaa^s \mid N(f)$ and $\Deltaa^{s+1} \nmid N(f)$. We will denote the integer $s$, called \emph{multiplicity of $\alpha$}, by $m_f(\alpha)$.
\end{definition}

In the case of $\alpha$ a real element, this condition is equivalent to $(w-\alpha)^s\mid f$ and $(w-\alpha)^{s+1}\nmid f$. If $\alpha$ is a spherical zero, then $\Deltaa$ divides $f$ and $\bar f$. Therefore $m_f(\alpha)$ is at least 2.
If $m_f(\alpha)=1$, $\alpha$ is  called a \emph{simple zero of $f$}. The reader observes that the multiplicity of $\alpha$ depends only on the conjugacy class $\isa$ of $\alpha$.

\begin{remark}
 The preceding definition is equivalent on quaternionic polynomials with the one given in \cite{BrW} and in \cite{GeSt3}. 
\end{remark}

%%%%%%%%%%%%%%

\section{Zeros of products: the ``camshaft effect''}
\label{sec:Zerosofproducts}

Let $f,g:B_R \lra \oc$ be regular functions. The aim of this section is to describe the zeros of $f \ast g$ in terms of the zeros of $f$ and of $g$. 
The following result is non--trivial and very important in the octonionic setting. It was proved by Ser\^odio (\cite[Theorem 10]{Se2}) in the particular case in which $g$ is a constant.

\begin{lemma}\label{Nfg}
It holds:
\[
N(f \ast g)=N(f)\, N(g).
\]\end{lemma}
\begin{proof}
Let $g$ be fixed. We must prove that the two real quadratic forms 
\[Q_1(f):=(f\ast g)\ast (\bar g\ast \bar f)\text{\quad and\quad} Q_2(f):=(f\ast \bar f)(g\ast \bar g)\]
coincide. By polarization, the equality of $Q_1$ and $Q_2$ is equivalent to the equality of the $\R$--bilinear symmetric forms
\[B_1(f_1,f_2):=(f_1\ast g)\ast(\bar g\ast \bar f_2)+(f_2\ast g)\ast(\bar g\ast \bar f_1),\]
\[B_2(f_1,f_2):=(f_1\ast \bar f_2)N(g)+(f_2\ast \bar f_1)N(g).\]
By linearity, it is sufficient to prove that $B_1$ and $B_2$ coincide for $f_1=w^ia$, $f_2=w^jb$, with $a,b\in\OO$. Since $B_\alpha(w^ia,w^jb)=w^{i+j}\ast B_\alpha(a,b)$ ($\alpha=1,2$), it suffices to prove that $B_1(a,b)=B_2(a,b)$  for every $a,b\in\OO$, which is equivalent to the equality $Q_1(a)=Q_2(a)$ for every $a\in\OO$. Now we can fix $a\in\OO$ and apply the same argument to the quadratic forms 
\[Q'_1(g):=(a\ast g)\ast (\bar g\ast \bar a)\text{\quad and\quad} Q'_2(g):=|a|^2 (g\ast \bar g).\]
We get that these two forms coincide if and only if $Q_1'(b)=Q'_2(b)$ for every $b\in\OO$. Therefore it suffices to prove the equality
\[(a\, b)(\bar b\, \bar a)=|a|^2\, |b|^2\text{\quad for every constant $a,b\in\OO$. }\]
In every alternative algebra the subalgebra generated by two elements is associative. Therefore $(a\, b)(\bar b\, \bar a)=a(b\bar b)\bar a=|a|^2\, |b|^2$.
 This completes the proof.
\end{proof}

\begin{corollary}\label{cor5}
It holds:
\[\bigcup_{\alpha\in V(f \ast g)}\is_{\alpha} \; \; = \, \, \bigcup_{\alpha\in V(f)\cup V(g)}\is_{\alpha}\]
or, equivalently, given any $\alpha \in \OO$, $V(f\ast g)\cap \is_{\alpha}$ is non--empty if and only if $\left(V(f)\cup V(g)\right)\cap\is_{\alpha}$ is non--empty. In particular, the zero set of $f\ast g$ is contained in the union of spheres $\bigcup_{\alpha\in V(f)\cup V(g)}\is_{\alpha}$.
\end{corollary}
\begin{proof}
By combining Corollary \ref{cor3} and Lemma \ref{Nfg}, we get:
\begin{eqnarray*}
\textstyle \bigcup_{\alpha\in V(f \ast g)}\is_{\alpha} &=& V(N(f \ast g)) = V(N(f))\cup V(N(g))=\\
 &=& \textstyle \left(\bigcup_{\alpha \in V(f)}\is_{\alpha}\right) \cup \left(\bigcup_{\alpha \in V(g)}\is_{\alpha}\right)=\bigcup_{\alpha\in V(f)\cup V(g)}\is_{\alpha}.
\end{eqnarray*}
Since $V(f \ast g) \subset V(N(f \ast g))$, the proof is complete.
\end{proof}

%\begin{corollary}
%If $f(\alpha)=0$, then  $N(f)=\Deltaa\, h$, with $h$ real.
%\end{corollary}
%\begin{proof}
%Since $f(\alpha)=0$, $f(w)=(w-\alpha)\ast g(w)$. Therefore $N(f)=N(w-\alpha)N(g)=\Deltaa N(g)$.
%\end{proof}

%??For simplicity, 
Now we give a precise description of the zeros of the product $f \ast g$. Thanks to Corollary \ref{cor5}, it is sufficient to analyze separately the spheres $\isa$ containing zeros of $f$ or of $g$.

Firstly, we consider the significant case: the non--real zeros.

Let $\alpha\in\oc$ be non--real. Suppose that
$\ra(f)(w)=wa+b$ and $\ra(g)(w)=wc+d$
for some $a,b,c,d \in \oc$. Then we have that
\begin{equation}\label{rfg}
\ra(f \ast g)(w)=w (ad+bc+t_{\alpha}ac)+bd-n_{\alpha}ac=(\ra(f)*\ra(g))(w)-\Deltaa ac.
\end{equation}
The latter assertion is an immediate consequence of the following fact
\begin{eqnarray*}
(wa+b) \ast (wc+d) &=& w^2(ac)+w(ad+bc)+bd=\\
&=& (\Deltaa+wt_{\alpha}-n_{\alpha})(ac)+w(ad+bc)+bd
%&=& \Deltaa \cdot (ac)+w (ad+bc+t_{\alpha}ac)+bd-n_{\alpha}ac.
\end{eqnarray*}

Thanks to Eq.\ (\ref{rfg}), we obtain:
 
\begin{theorem}[the camshaft effect] \label{thm:cam-shaft}
The following statements hold.
\begin{itemize}
 \item[$(1)$] Suppose $f$ has an isolated zero $\alpha$ in $\isa$ of multiplicity $s$ and $g$ is nowhere zero on~$\isa$. Then $f \ast g$ has an isolated zero $\alpha^{\pr}$ in~$\isa$ of multiplicity $s$, $ad+(\bar{\alpha} a)c \neq 0$ and it holds:
$$
\alpha^{\pr}=\left((\alpha a)d+n_{\alpha}ac\right)\left(ad+(\bar{\alpha} a)c\right)^{-1}.
$$
In particular, if $(\alpha a)d=\alpha (ad)$ and $n_{\alpha}ac=\alpha \left((\bar{\alpha}a)c\right)$ $($for example,  if $a$ is real or if $a,b,c,d$ belong to an associative subalgebra of\, $\OO\,)$,
%$($for example, in the quaternionic case$)$, 
then $\alpha^{\pr}=\alpha$.
 \item[$(2)$] Suppose $f$ is nowhere zero on~$\isa$ and $g$ has an isolated zero $\beta$ in $\isa$ of multiplicity~$t$. Then $f \ast g$ has an isolated zero $\beta^{\pr}$ in~$\isa$ of multiplicity $t$, $bc+a(\bar{\beta} c) \neq 0$ and it holds:
$$
\beta^{\pr}=\left(b(\beta c)+n_{\alpha}ac\right)\left(bc+a(\bar{\beta}c)\right)^{-1}.
$$
In particular, if $b(\beta c)=(b \beta)c$ and $a(\bar \beta c)=(a \bar \beta) c$ $($for example, if $c$ is real or if $a,b,c,d$ belongs to an associative subalgebra of\, $\OO\,)$, then we have that $\beta'=(b+a\bar{\beta}\, ) \beta (b+a\bar{\beta} \, )^{-1}$.
%In particular, in the quaternionic case, we have that $\beta^{\pr}=\left(b+a\bar{\beta} \, \right) \beta \left(b+a\bar{\beta} \, \right)^{-1}$.
 \item[$(3)$] Suppose $f$ has an isolated zero $\alpha$ in~$\isa$ of multiplicity $s$ and $g$ has an isolated zero $\beta$ in $\isa$ of multiplicity~$t$. Then, if $a(\beta c)=(\bar{\alpha}a)c$, then $\alpha$ is a spherical zero of $f \ast g$ of multiplicity $s+t$. Otherwise, $f \ast g$ has an isolated zero $\gamma$ in~$\isa$ of multiplicity $s+t$ such that
$$
\gamma=\left(-(\alpha a)(\beta c)+n_{\alpha}ac\right)\left(-a(\beta c)+(\bar{\alpha}a)c\right)^{-1}.
$$
\item[$(4)$] Let $\alpha$ be a spherical zero of $f$ or of $g$, and let $s, t$ be the respective multiplicities of $\alpha$ $($possibly zero$)$. Then $\alpha$ is a spherical zero of $f \ast g$ of multiplicity $s+t$.
\end{itemize}
\end{theorem}

\begin{proof}
In the following, we will use the fact that, in any alternative quadratic algebra, the trace is associative and commutative (cf.\ e.g.\  \cite[\S9.1.2]{numbers}).
 
First, we prove the theorem without the part concerning multiplicities.

(1)\ In this case $a\ne0$ and $\alpha a+b=0$ (cf.\ Corollary \ref{cor2}(2)(ii)). If $c=0$, then $d\ne0$ (otherwise $g$ would vanish on the whole sphere $\isa$).  Then $ad+(\bar\alpha a)c=ad\ne0$ and Eq.\ \eqref{rfg} gives the isolated zero $\alpha'=-(bd)(ad)^{-1}=((\alpha a)d)(ad)^{-1}$  of $f\ast g$. 
If $c\ne0$, then $ad+(\bar\alpha a)c$ cannot vanish. In fact, if it were $ad+(\bar\alpha a)c=0$, then we would have that $\bar\alpha=-((ad)c^{-1})a^{-1}$. But then $\tr({\alpha})=\tr(-dc^{-1})$ and $|\alpha|=|-dc^{-1}|$. This would mean that $\beta=-dc^{-1}$ belongs to $\isa\cap V(g)$, contradicting our assumptions. 

Corollary \ref{cor5} tells that $\isa$ must contain a zero of $f\ast g$. From Corollary \ref{cor2}(2)(ii) and Eq.\ \eqref{rfg}, this zero must be 
\[\alpha'=(-bd+n_{\alpha}ac)(ad+bc+t_{\alpha}ac)^{-1}.\]
Note that the octonion $ad+bc+t_{\alpha}ac$\, does not vanish, since it is equal to $ad-(\alpha a)c+(t_{\alpha}a)c=ad-(\alpha a)c+(\alpha a)c+(\bar\alpha a)c=ad+(\bar\alpha a)c$. We conclude that $\alpha'=((\alpha a)d+n_{\alpha}ac)(ad+(\bar\alpha a)c)^{-1}.$

(2)\ This second case is similar to the first one. Now $c\ne0$ and $\beta c+d=0$. If $a=0$, then $b\ne0$ and $bc+a(\bar{\beta}c)=bc\ne0$.
From Eq.\ \eqref{rfg} we get the isolated zero 
$\beta'=(b(\beta c))(bc)^{-1}.$ If $a$  does not vanish, then it still holds $bc+a(\bar\beta c)\ne0$. In fact, if this were not true, we would have that $\beta$ is conjugate to $-ba^{-1}$ and therefore $-ba^{-1}\in \isa \cap V(f)$.
By using Corollary \ref{cor5}, Corollary \ref{cor2}(2)(ii), $\beta c=-d$ and $t_\alpha=t_\beta$, we obtain the unique zero
\[\beta'=(-bd+n_{\alpha}ac)(ad+bc+t_{\alpha}ac)^{-1}=(b(\beta c)+n_{\alpha}ac)(bc+a(\bar{\beta}c))^{-1}\]
of $f \ast g$ in $\isa$.

(3)\ In this case, $a,c\ne0$, $b=-\alpha a$, $d=-\beta c$. From Eq.\ \eqref{rfg}, we get the remainder:
\begin{align*}
\ra(f \ast g)(w)&= w (-a(\beta c)-(\alpha a)c+t_{\alpha}ac)+(\alpha a)(\beta c)-n_{\alpha}ac\\
& =w (-a(\beta c)+(\bar\alpha a)c)+(\alpha a)(\beta c)-n_{\alpha}ac.
\end{align*}
If $a(\beta c)=(\bar\alpha a)c$, then $\ra(f \ast g)$ is constant. On the other hand, by Corollary \ref{cor5}, $\ra(f \ast g)$ vanishes on some point of $\isa$ and hence it is null. It follows that $f \ast g$ has $\alpha$ as a spherical zero.
Instead, if $a(\beta c)\ne(\bar\alpha a)c$, then $f\ast g$ has the isolated zero
\[\gamma=\left(-(\alpha a)(\beta c)+n_{\alpha}ac\right)\left(-a(\beta c)+(\bar{\alpha}a)c\right)^{-1}.\]

(4)\ If the real polynomial $\Deltaa$ divides $f$ or $g$, then it divides $f\ast g$. The conclusion follows from Corollary \ref{cor2}(2)(i).

Now we consider multiplicities. If $\alpha$ is a zero of $f$ of multiplicity $s$ and $\beta\in\isa$ is a zero of $g$ of multiplicity $t$, then $N(f)=\Deltaa^s h_1$ and $N(g)=\Deltaa^t h_2$, where $h_1$ and $h_2$ real regular functions such that $\ra(h_1)=wa+b\ne0$ and $\ra(h_2)=wc+d\ne0$. From Lemma \ref{Nfg}, it follows that $N(f\ast g)=\Deltaa^{s+t}(h_1h_2)$. We have to prove that $\Deltaa^{s+t+1}\nmid N(f\ast g)$. By a density argument, we see that this is equivalent to show that $\Deltaa\nmid h_1h_2$, i.e.\ that the remainder $\ra(h_1h_2)$ is not zero. From Eq.\ \eqref{rfg}, $\ra(h_1h_2)$ is zero if and only if
\begin{equation}
ad+bc+t_{\alpha}ac=bd-n_{\alpha}ac=0.
	\label{eq9}
\end{equation} 
Since $a,b,c,d$ are real, Eq.\ \eqref{eq9} implies that $a(d+\alpha c)(d+\bar\alpha c)=0$. But this cannot happen since $\alpha$ is non--real and $a\ne0$.
\end{proof}

\begin{example}
For every non--real $\alpha\in\oc$, $\bar\alpha$ is conjugated to $\alpha$. Let $a$ be such that $\bar\alpha=a\alpha a^{-1}$. If $f(w)=wa-\alpha a$, $g(w)=w-\alpha$, then $f\ast g=w^2a-w(a\alpha+\alpha a)+\alpha a\alpha$ and $V(f)=V(g)=\{\alpha\}$ while  $V(f\ast g)=\isa$ (see Theorem \ref{thm:cam-shaft}(3)).
\end{example}

\begin{remark}
In the octonionic case, the camshaft effect appears even if one of the functions is constant: the zeros of $f$ and of $f\ast c$ ($c\in\OO$ a constant) can be different (cf.\ Ser\^odio \cite{Se1} when $f$ is a polynomial). For example, let $f(w)=wi-j$, $g(w)=k$. Then $f\ast g=w(ik)-jk$ has $-ij$ as unique zero, while $V(f)=\{ij\}$.
\end{remark}

It remains to consider real zeros.

\begin{proposition}\label{realzeros}
Let $\alpha$ be a real octonion. Suppose that $\alpha$ is a zero of $f$ of multiplicity~$s$ and a zero of $g$  of multiplicity~$t$ $(\, s$ and $t$ possibly zero$)$. Then $\alpha$ is a zero of $f \ast g$ of multiplicity $s+t$. 
\end{proposition}
\begin{proof}
If the real polynomial $w-\alpha$ divides $f$ or $g$, then it divides $f\ast g$. In particular, $\ra(f\ast g)=0$. Moreover, if $f=(w-\alpha)^s h_1$ and $g=(w-\alpha)^t h_2$ with $h_1(\alpha) \neq 0$ and $h_2(\alpha) \neq 0$, then $f\ast g=(w-\alpha)^{s+t} (h_1 \ast h_2)$. It remains to prove that $(w-\alpha)^{s+t+1}\nmid f\ast g$. By density, this is equivalent to show that $w-\alpha\nmid h_1\ast h_2$, i.e.\ $(h_1\ast h_2)(\alpha)\ne0$. But $(h_1\ast h_2)(\alpha)=0$ would imply that $N(h_1\ast h_2)(\alpha)=N(h_1)(\alpha)\cdot N(h_2)(\alpha)=0$. On the other hand, since $\alpha$ is real, $N(h_1)(\alpha)=|h_1(\alpha)|^2 \neq 0$ and $N(h_2)(\alpha)=|h_2(\alpha)|^2 \neq 0$, which is a contradiction.
\end{proof}

%%%%%%%%%%%%%%%%%%%%%

\section{Applications}
\label{sec:Applications}

We begin with an octonionic version of the fundamental theorem of algebra.

\begin{theorem}[Fundamental theorem of algebra]\label{TFA}
If $f$ is a regular octonionic polynomial of degree $n>0$, then its zero set $V(f)$ is non--empty and the cardinality $k$ of the distinct conjugacy classes of elements of $V(f)$ is less than or equal to $n$. More precisely, if $\alpha_1,\ldots,\alpha_r$ are the distinct real zeros of $f$, $\alpha_{r+1},\ldots,\alpha_{r+i}$ are the distinct isolated zeros of $f$ and $\alpha_{r+i+1},\ldots, \alpha_ {r+i+s}$ are pairwise non--conjugate spherical zeros of $f$ such that $\bigcup_{j=1}^s \is_{\alpha_{r+i+j}}$ is the set of all spherical zeros of $f$, then $k=r+i+s$ and the following equality holds:
\[
\sum_{j=1}^k m_f(\alpha_j)=n.
\]
In particular, we have that $r+i+2s \leq n$.
\end{theorem}
\begin{proof}
The normal polynomial $N(f)$ is a real polynomial of degree $2n$. Let $I\in\is$. Then the set $V_I(N(f)):=\{z\in\C_I\ |\ N(f)(z)=0\}=\C_I \cap \bigcup_{\alpha \in V(f)}\isa$ is non--empty and contains at most $2n$ elements. Corollary \ref{cor3} tells that $V(f)\cap\isa\ne\emptyset$ for every $\alpha$ such that $V_I(N(f))\cap\isa\ne\emptyset$. Therefore, $V(f)$ is non--empty.

If $f=\Deltaa h+wa+b$ for some regular function $h$ and $a,b \in \OO$, then $\bar f=\Deltaa \bar h+w\bar a+\bar b$. This implies that an isolated zero $\alpha=-ba^{-1}$ of $f$, with $m_f(\alpha)=s$, corresponds to an isolated zero $\hat{\alpha}=-\bar b\bar a^{-1} \in \isa$ of $\bar f$ with the same multiplicity. From Theorem \ref{thm:cam-shaft}(3), we get that $\alpha$ is a spherical zero of $N(f)=f\ast \bar f$ of multiplicity $2s$. The same property holds for spherical and real zeros of $f$. Then, given $\alpha_1,\ldots,\alpha_k$ as in the statement of the theorem, it follows that
\[
2\sum_{j=1}^km_{f}(\alpha_j)=\sum_{j=1}^km_{N(f)}(\alpha_j)=2n.
\]
Since the multiplicity of a spherical zero is at least $2$, it follows immediately that $r+i+2s \leq n$. The proof is complete.
\end{proof}

A repeated application of the division lemma (see Remark \ref{rem:pol}), of Theorem \ref{thm:cam-shaft} and of Theorem \ref{TFA} gives:

\begin{corollary}[Factorization lemma]\label{factlemma}
Let $f$ be a regular octonionic polynomial of degree $n>0$. Let $\alpha_1\in V(f)$. Then there exist $\alpha_2,\ldots,\alpha_n, c \in \oc$ with $c \neq 0$ such that $f$ factors as follows:
%\[f(w)=(w-\alpha_1)\ast((w-\alpha_2)\ast(\cdots\ast(w-\alpha_d)\cdots)).\]
%\[f(w)=(w-\alpha_1)\ast f_1(w),\ f_1(w)=(w-\alpha_2)\ast f_2(w),\ldots, f_{d-1}(w)=w-\alpha_d.\]
\begin{align*}
f(w)&=(w-\alpha_1)\ast f_1(w),\text{\quad where}\\
f_k(w):&=(w-\alpha_{k+1})\ast f_{k+1}(w)\text{\quad for $k=1,2,\ldots, n-2$\ and}\\
f_{n-1}(w):&=(w-\alpha_n)\ast c.
\end{align*}
Moreover, for every $k=1,\ldots, n-1$, $\alpha_{k+1}$ is a zero of $f_k$ and it is conjugate to a zero  $\alpha'_{k+1}$ of $f$. \qed
\end{corollary}

\begin{remark}
Theorem \ref{thm:cam-shaft} tells how to obtain the zeros $\alpha'_{k+1}\in V(f)$ from the set $\{\alpha_{i}\}_{i=2,\ldots, n}$.
\end{remark}

A standard application of the fundamental theorem of algebra in the complex field is the uniqueness of monic polynomials with prescribed zeros. On the quaternions and on the octonions, a new phenomenon appears. 
The multiplici\-ty of an isolated zero of the sum of two regular functions can be strictly less than those of the functions. This fact implies that two different monic polynomials can have the same zeros with the same multiplicities. For example, the polynomials $f(w)=(w-i)\ast (w-i)$ and $g(w)=(w-i)\ast (w-j)$ 
%($i,j\in\is$, $i\ne \pm j$) 
have the unique isolated zero $w=i$ with multiplicity 2, while the difference $f-g$ has a simple zero $i$.
This fact cannot happen when all the zeros are real or spherical or simple, since in this case
\[m_{f+g}(\alpha)\ge \min\{m_f(\alpha),m_g(\alpha)\}.\]

On the quaternions, the existence of a unique monic regular polynomial with assigned pairwise non--conjugate zeros was proved by Beck \cite{Be} (see also \cite{CeM}, \cite{BrW} and \cite{GeSt3}). On the octonions, it was recently proved by Ser\^odio \cite{Se2}. The existence of an infinite numbers of monic regular quaternionic polynomials with prescribed multiple isolated zeros was proved by Beck \cite{Be}.
%Note that the factorization lemma above (Corollary \ref{factlemma}), together with the formulas of the ``camshaft effect'', gives an iterative method to construct an octonionic polynomial with a prescribed set of zeros.

An immediate consequence of the fundamental theorem and Propositions \ref{pro3}, \ref{pro4} is the following result, which generalizes the one proved by Lam in the quaternionic case (cf.\ \cite[\S16.19]{Lam}):

\begin{corollary}\label{cor12}
Let $f(w)=\sum_{i=0}^n w^ia_i$ be a regular octonionic polynomial of degree $n>0$. The following statements hold.
\begin{itemize}
\item[$(1)$] If the coefficients $a_1,\ldots, a_n$ are real and $a_0$ is non--real, then $f$ has only isolated zeros. 
\item[$(2)$] If $a_2,\ldots, a_n$ are real and either $a_0$ or $a_1$ is non--real, then $f$ has no spherical zeros.
\end{itemize}
In particular, the set $V(f)$ contains exactly $n$ elements, counted with their multiplicities.\qed
\end{corollary}

\begin{example}
To illustrate the situation described by the preceding corollary, we consider the polynomial $f(w)=w^2+wi+j$. Since $N(f)=w^4+w^2+1$, $\bigcup_{\alpha \in V(f)}\isa$ is equal to $\is_{\beta} \cup \is_{\gamma}$, where $\beta:=\frac12+i\frac{\sqrt3}2$ and $\gamma:=-\frac12+i\frac{\sqrt3}2$.
The corresponding remainders are
\[r_{\beta}(f)=w(1+i)-(1-j) \text{\quad and \quad} r_{\gamma}(f)=w(-1+i)-(1-j),\]
which give the following two simple isolated zeros of $f$: $\alpha_1=\frac12(1-i-j-ij)\in\is_{\beta}$ and $\alpha_2=\frac12(-1-i+j-ij)\in\is_{\gamma}$.
\end{example}

%%%%%%%%%%%%%%%%%%%%%%%%%%%%%%%%%
%%%%%%%%%%%%%%%%%%%%%%%%%%%%%%%%%


\begin{thebibliography}{00}
% please try to use the bibitem system -
% the references should be in alphabetical order of authors' names.
% Articles with a single author first, author will 1 co-author next,
% then author with several co-authors;


% \bibitem{label}
% Text of bibliographic item
\bibitem{Be}
B. Beck,  \emph{Sur les équations polynomiales dans les quaternions}. 
Enseign.\ Math.\ (2) 25 (1979), no. 3-4, 193--201. 

\bibitem{BrW}
U.\ Bray, G.\ Whaples, \emph{Polynomials with coefficients from a division ring.}
Canad.\ J.\ Math.\ 35 (1983), no. 3, 509--515. 

\bibitem{CeM}
L. Cerlienco, M. Mureddu, \emph{A note on polynomial equations in quaternions.} 
Rend.\ Sem.\ Fac.\ Sci.\ Univ.\ Cagliari 51 (1981), no. 1, 95--99. 

\bibitem{CoSaSt1}
F.\ Colombo, I.\ Sabadini, D.C.\ Struppa, \emph{Slice monogenic functions.}
Israel J.\ Math., 171 (2009), 385--403.

%\bibitem{CoSaSt2}
%F. Colombo, I. Sabadini, D.C. Struppa, A structure formula for slice monogenic functions and some of its consequences, in {Hypercomplex analysis} (eds. I. Sabadini, M. Shapiro and F. Sommen), Birkh\"auser, Basel, 2009, pp.235--258.

%\bibitem{Cu}
%C.G. Cullen, An integral theorem for analytic intrinsic functions on quaternions, Duke Math.\ J., 32 (1965), 139--148.

\bibitem{DaN}
B.\ Datta, S.\ Nag, \emph{Zero-sets of quaternionic and octonionic analytic functions with central coefficients.}  Bull.\ London Math.\ Soc.\  19  (1987),  no. 4, 329--336. 

\bibitem{numbers}
H.-D.\ Ebbinghaus,  H.\ Hermes, F.\ Hirzebruch, M.\ Koecher, K.\ Mainzer, J.\ Neukirch, A.\ Prestel, R.\ Remmert, Numbers. %With an introduction by K. Lamotke. Translated from the second 1988 German edition by H.\ L.\ S. Orde.  
Springer-Verlag, New York, 1991.

\bibitem{EiNi}
S.\ Eilenberg, I.\ Niven, \emph{The ``fundamental theorem of algebra'' for quaternions.}  Bull.\ Amer.\ Math.\ Soc.\  50,  (1944), 246--248.

\bibitem{EiSt}
S.\ Eilenberg, N.\ Steenrod, Foundations of algebraic topology. Princeton University Press, Princeton, New Jersey, 1952. 

\bibitem{GRW}
I.\ Gelfand, V.\ Retakh, R.\ Wilson, \emph{Quadratic-linear algebras associated with decompositions
of noncommutative polynomials and differential polynomials.} Selecta Math.\ (N.S.) 7, 493--523 (2001).

\bibitem{GGRWL}
I.\ Gelfand, S.\ Gelfand, V.\ Retakh, R.L.\ Wilson, \emph{Factorizations of polynomials over noncommutative algebras and sufficient sets of edges in directed graphs.}  Lett.\ Math.\ Phys.\  74  (2005),  no. 2, 153--167. 

\bibitem{GeSto1}
G.\ Gentili, C.\ Stoppato, \emph{Zeros of regular functions and polynomials of a quaternionic variable}. Michigan Math.\ J., 56 (2008), 665--667.

%\bibitem{GeSto2}
%G.\ Gentili, C.\ Stoppato, The open mapping theorem for quaternionic regular functions, to appear in: Ann.\ Sc.\ Norm.\ Super.\ Pisa Cl.\ Sci.\ (5).

\bibitem{GeStoStVl}
G.\ Gentili, C.\ Stoppato, D.\ C.\ Struppa, and F.\ Vlacci. \emph{Recent developments
for regular functions of a hypercomplex variable.} In {Hypercomplex analysis},  eds. I. Sabadini, M. Shapiro and F. Sommen, Trends Math., Birkh\"auser, Basel, 2009, pp.\ 165--186.

\bibitem{GeSt1}
G.\ Gentili, D.\ Struppa,\emph{ A new approach to Cullen-regular functions of a quaternionic
variable}. C.R.\ Acad.\ Sci.\ Paris, 342 (2006), 741--744.

\bibitem{GeStPreprint}
G.\ Gentili, D.\ Struppa, \emph{Regular functions on the space of Cayley numbers.} To appear in Rocky Mountain J.\ Math.
%Preprint, Dipartimento di Matematica ``U.Dini'', Universit\`a di Firenze, n.13 (2006). 

\bibitem{GeSt2}
G.\ Gentili, D.\ Struppa, \emph{A new theory of regular functions of a quaternionic variable}. Adv.\
Math.\ 216 (2007), 279--301.

\bibitem{GeSt3}
G.\ Gentili, D.\ Struppa, \emph{On the multiplicity of zeroes of polynomials with quaternionic coefficients}. Milan J.\ Math., 76 (2008), 15--25.

\bibitem{GeSt4}
G.\ Gentili, D.\ Struppa, \emph{Regular functions on a Clifford Algebra}. Complex Var.\ Elliptic Equ.\ 53 (2008), 475--483.

\bibitem{GeStVl}
G.\ Gentili, D.\ Struppa, F.\ Vlacci, \emph{The fundamental theorem of algebra for Hamilton and Cayley numbers}. Math.\ Z., 259 (2008), 895-902. 

%\bibitem{GeVl}
%G.\ Gentili, F.\ Vlacci, Rigidity for regular functions over Hamilton and Cayley numbers and a boundary Schwarz lemma, to appear in: Indag.\ Math.\ (N.S.).

\bibitem{GHS}
K.\ G\"urlebeck,  K.\ Habetha and W.\ Spr\"ossig,   {Holomorphic functions in the plane and $n$-dimensional space.}
Translated from the 2006 German original {Funktionentheorie in Ebene und Raum}, Birkh\"auser Verlag, Basel, 2008.

\bibitem{GoM}
B.\ Gordon, T.S.\ Motzkin, \emph{On the zeros of polynomials over division rings}. Trans.\ AMS 1965, 116, 218--226.

%\bibitem{Ha}
%F.R.\ Harvey, Spinors and calibrations. Academic Press, Inc., Boston, MA, 1990.

\bibitem{Jou}%Jou, Yuh-Lin, 
Y.\ Jou, \emph{The ``fundamental theorem of algebra'' for Cayley numbers}. Acad.\ Sinica Science Record 3, (1950), 29--33. 

\bibitem{Lam}
T.Y.\ Lam, A first course in noncommutative rings. Springer-Verlag, New York, 1991.

\bibitem{Ni}
I.\ Niven, \emph{Equations in quaternions}.  Amer.\ Math.\ Monthly  48,  (1941), 654--661. 
%The roots of a quaternion. Amer.\ Math.\ Monthly 49, (1942). 386--388. 

\bibitem{PoSh}
A.\ Pogorui, M.V.\ Shapiro, \emph{On the structure of the set of zeros of quaternionic polynomials}.
Complex Variables 49, (2004) no. 6, 379--389.

%\bibitem{PuWa}
%S.\ Pumpl\"un, S.\ Walcher,  On the zeros of polynomials over quaternions. Comm.\ Algebra 30 (2002), no. 8, 4007--4018. 

\bibitem{RoO}
H.\ Rodríguez-Ordóñez, \emph{A note on the fundamental theorem of algebra for the octonions}. Expo.\ Math.\ 25 (2007), no. 4, 355--361. 

\bibitem{Sch}
R.D.\ Schafer, An introduction to nonassociative algebras.  Academic Press, New York-London 1966.

\bibitem{Se1} R.\ Ser\^odio, \emph{On octonionic polynomials}. Adv.\ Appl.\ Clifford Alg.\ 17 (2007), 245--258. 

\bibitem{Se2} R.\ Ser\^odio, \emph{Construction of octonionic polynomials}. Adv.\ Appl.\ Clifford Alg.\  published online,  (2008), DOI 10.1007/s00006-008-0140-5. 

\bibitem{SeSiu}
R.\ Ser\^odio, L.\ Siu, \emph{Zeros of quaternion polynomials}.  Appl.\ Math.\ Lett.\  14  (2001),  no. 2, 237--239.

\bibitem{Su} 
A.\ Sudbery, \emph{Quaternionic analysis}. {Mat.\ Proc.\ Camb.\ Phil.\ Soc.} {85}  (1979), 199--225.

\bibitem{To}
N.\ Topuridze, \emph{On the roots of polynomials over division algebras}.  Georgian Math.\ J.\  10  (2003),  no. 4, 745--762.



\end{thebibliography}
\end{document}